\documentclass[reqno]{amsart}
\usepackage{amsthm,amsmath,amssymb}
\usepackage{stmaryrd,mathrsfs}
\usepackage[T1]{fontenc}
\usepackage{esint}
\usepackage{tikz}

\allowdisplaybreaks

\newcommand{\ud}[0]{\,\mathrm{d}}

\newcommand{\abs}[1]{|#1|}
\newcommand{\Babs}[1]{\Big|#1\Big|}

\newcommand{\Norm}[2]{\|#1\|_{#2}}

\newcommand{\pair}[2]{\langle #1,#2 \rangle}

\newcommand{\ave}[1]{\langle #1\rangle}

\newcommand{\supp}[0]{\operatorname{supp}}

\newcommand{\R}{\mathbb{R}}

\newcommand{\N}{\mathbb{N}}
\newcommand{\Z}{\mathbb{Z}}

\swapnumbers \numberwithin{equation}{section}

\theoremstyle{plain}
\newtheorem{theorem}[equation]{Theorem}
\newtheorem{proposition}[equation]{Proposition}

\newtheorem{lemma}[equation]{Lemma}

\theoremstyle{definition}

\theoremstyle{remark}
\newtheorem{remark}[equation]{Remark}

\makeatletter
\@namedef{subjclassname@2010}{%
  \textup{2010} Mathematics Subject Classification}
\makeatother

\begin{document}

\title{The $A_2$ theorem: Remarks and complements}

\author[T.~P.\ Hyt\"onen]{Tuomas P.\ Hyt\"onen}
\address{Department of Mathematics and Statistics, P.O.B.~68 (Gustaf H\"all\-str\"omin katu~2b), FI-00014 University of Helsinki, Finland}
\email{tuomas.hytonen@helsinki.fi}

\date{\today}

\thanks{The author is supported by the European Union through the ERC Starting Grant ``Analytic--probabilistic methods for borderline singular integrals'', and by the Academy of Finland, grants 130166 and 133264.}



\maketitle

\begin{abstract}
I give a mini-survey of several approaches to the $A_2$ theorem, biased towards the ``corona'' rather than the ``Bellman'' side of the coin. There are two new results (a streamlined form of Lerner's local oscillation formula, and the sharpness of the linear-in-complexity weak $(1,1)$ bound for dyadic shifts) and two new proofs of known results (the $A_p$--$A_\infty$ testing conditions, and the two-weight $T1$ theorem for positive dyadic operators).
\end{abstract}

\section{Introduction}

$A_2$ theorem for a Calder\'on--Zygmund operator $T$ is the sharp weighted bound
\begin{equation}\label{eq:A2}
  \textup{``$A_2$ for $T$'':}\qquad\Norm{Tf}{L^2(w)}\leq c_T[w]_{A_2}\Norm{f}{L^2(w)}
\end{equation}
in terms of the Muckenhoupt ``norm'' (or ``characteristic'', or ``constant'')
\begin{equation*}
  [w]_{A_2}:=\sup_Q\ave{w}_Q\ave{\frac{1}{w}}_Q:=\sup_Q\fint_Q w\cdot\fint_Q \frac{1}{w}
  :=\sup_Q\frac{\int_Q w}{\abs{Q}}\cdot\frac{\int_Q 1/w}{\abs{Q}}.
\end{equation*}
After being established for several particular operators first \cite{HLRSUV,Petermichl:Hilbert,PV:Beurling,Vagharshakyan}, I proved \eqref{eq:A2} for an arbitrary Calder\'on--Zygmund operator $T$ in July 2010 \cite{Hytonen:A2}. This was considered quite a difficult result back then. It used, in particular: a weighted $T1$ theorem of P\'erez, Treil and Volberg \cite{PTV}, permitting the reduction of \eqref{eq:A2} to the ``testing condition''
\begin{equation}\label{eq:testing}
  \textup{``testing $T$'':}\qquad 
  \begin{cases} \Norm{T(w^{-1}1_Q)}{L^2(w)}\leq c_T[w]_{A_2}w^{-1}(Q)^{1/2},  \\
     \Norm{T^*(w1_Q)}{L^2(w^{-1})}\leq c_T[w]_{A_2}w(Q)^{1/2}; \end{cases}
\end{equation}
a reduction of the operator $T$ to dyadic model operators called shifts $S_k$ via a probabilistic argument inspired by the work of  Nazarov, Treil and Volberg on non-doubling harmonic analysis \cite{NTV:Tb}; and  a subtle multi-step (``corona'') decomposition, elaborating on earlier work of Lacey, Petermichl and Reguera \cite{LPR}.
(Personally, I think that ``corona'' is a misused word in this context, since the original Corona Problem is rather distant, but I adopt this common terminology for this article.)

The last two years have greatly expanded our understanding of the $A_2$ theorem, and several known approaches are illustrated in the following diagram. The nodes ``$A_2$ for $T$'' and ``testing $T$'' (for different choices of $T$) represent intermediate results, as defined in \eqref{eq:A2} and \eqref{eq:testing}, whereas the arrows indicate different routes of passing from one intermediate result to the next. Whenever an arrow crosses a dashed line, it means that a corresponding auxiliary result is needed at that point. The further right one applies the weighted $T1$ theorem, the more difficult it becomes. To a smaller extent, this is also true for the corona decomposition.

The dotted lines indicate steps that are possible but unnecessary, since there is also a short direct proof of ``$A_2$ for $S_0^+$''; but I will say more about this path below.

\noindent
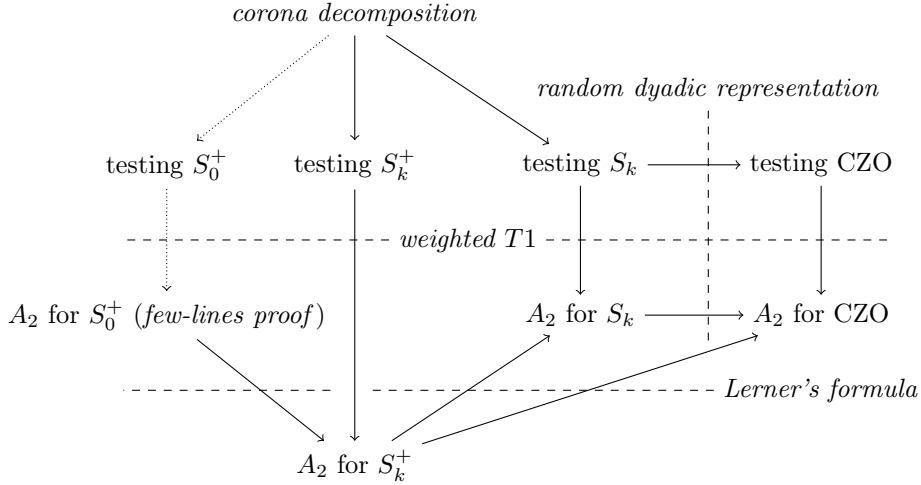
\begin{tikzpicture}
  \node (A2T) {$A_2$ for CZO};
  \node (testT) [node distance=2cm, above of=A2T] {testing CZO};
  \node (A2Sk) [node distance=3.2cm, left of=A2T] {$A_2$ for $S_k$};
  \node (testSk) [node distance=2cm, above of=A2Sk] {testing $S_k$};
  \node (testSp) [node distance=3cm, left of=testSk] {testing $S_k^+$};
  \node (A2Sp) [node distance=4cm, below of=testSp] {$A_2$ for $S_k^+$};
  \node (A2S0) [node distance=5.5cm, left of=A2Sk] {$A_2$ for $S_0^+$ (\emph{few-lines proof})};
  \node (testS0) [node distance=2cm, above of=A2S0] {testing $S_0^+$};
  \node (Cor) [above of=testSp, node distance=2cm] {\emph{corona decomposition}};
  \draw[->] (testT) to node {} (A2T);
  \draw[->] (testSk) to node {} (A2Sk);
  \draw[->] (testSp) to node {} (A2Sp);
  \draw[->] (testSk) to node {} (testT);
  \draw[->] (A2Sk) to node {} (A2T);
  \draw[->] (A2Sp) to node {} (A2Sk); 
  \draw[->] (A2S0) to node {} (A2Sp);
  \draw[->] (A2Sp) to node {} (A2T);
  \draw[->] (Cor) to node{} (testSp);
  \draw[->] (Cor) to node{} (testSk);
  \draw[->, densely dotted] (Cor) to node{} (testS0);
  \draw[->, densely dotted] (testS0) to node{} (A2S0);
  \node (A1) [left of=testT, node distance=1.5cm] {};
  \node (A2) [left of=A2T, node distance=1.5cm] {};
  \node (A3) [above of=A1, node distance=1cm] {\emph{random dyadic representation}};
  \node (A4) [below of=A2, node distance=0.5cm] {};
  \draw[-, dashed] (A3) to node {} (A4);
  \node (B1) [below of=testSp, node distance=1cm] {};
  \node (B2) [below of=testT, node distance=1cm] {};
  \node (B3) [right of=B1, node distance=1.5cm] {\emph{weighted $T1$}};
  \node (B4) [right of=B2, node distance=1cm] {};
  \node (B5) [left of=B1, node distance=3.2cm] {};
  \draw[-, dashed] (B3) to node {} (B4);
  \draw[-, dashed] (B3) to node {} (B5);
  \node (L) [below of=A2T] {\emph{Lerner's formula}};
  \node (L1) [left of=L, node distance=6.1cm] {};
  \node (L2) [left of=L1, node distance=0.3cm] {};
  \node (L3) [left of=L2, node distance=3cm] {};
  \draw[-, dashed] (L) to node {} (L1);
  \draw[-, dashed] (L2) to node {} (L3);
\end{tikzpicture}

In the diagram, CZO denotes an arbitrary Calder\'on--Zygmund operator, and $S_k$ an arbitrary dyadic shift of order $k\in\N$, as defined in \cite{Hytonen:A2,LPR}.  In the conditions ``$A_2$ for $S_k$'' and ``testing $S_k$'', as defined in \eqref{eq:A2} and \eqref{eq:testing}, it is understood that $c_{S_k}$ should grow at most polynomially in $k$. For the purposes of the present discussion, it is not necessary to recall the general definition of a dyadic shift, since we only explicitly deal with the following particular case: The symbol $S_k^+$ (of which $S_0^+$ is a special case) denotes a positive dyadic shift of order $k$ of the specific form
\begin{equation*}
  S_k^+ f=\sum_{K\in\mathscr{K}} 1_K\fint_{K^{(k)}}f\ud x,
\end{equation*}
where $K^{(k)}$ is the $k$ generations older dyadic ancestor of $K$ (so that $K^{(0)}:=K$), and $\mathscr{K}$ is an arbitrary sparse collection of dyadic cubes, i.e., there are pairwise disjoint subsets $E(K)\subset K$ with $\abs{E(K)}\geq c\abs{K}$ for a fixed constant $c>0$. 

My original proof of the $A_2$ theorem \cite{Hytonen:A2} proceeded via the ``top right'' route
 \begin{equation*}
  \textup{corona}\quad\to\quad\textup{testing }S_k\quad\to\quad\text{testing CZO}\quad\to\quad A_2\text{ for CZO},
\end{equation*}
where the last step was borrowed from P\'erez, Treil and Volberg \cite{PTV}.
This difficult step was avoided by the somewhat easier route
 \begin{equation}\label{eq:HPTV}
  \textup{corona}\quad\to\quad\textup{testing }S_k\quad\to\quad A_2\textup{ for }S_k\quad\to\quad A_2\text{ for CZO}.
\end{equation}
taken by Hyt\"onen--P\'erez--Treil--Volberg \cite{HPTV}. The estimates along this route were further elaborated by Hyt\"onen, Lacey, Martikainen et al. \cite{HLMORSU} to show that even the maximal truncated singular integrals
\begin{equation*}
  T_{\#}f(x):=\sup_{\epsilon>0}\abs{T_\epsilon f(x)},\qquad
  T_\epsilon f(x):=\int_{\abs{x-y}>\epsilon}K(x,y)f(y)\ud y,
\end{equation*}
can be reached, proving ``$A_2$ for $T_{\#}$.'' It was at this point that the $A_2$ technology was at the peak of its difficulty: In addition to the methods shown in the diagram, ideas coming from the proof of Carleson's theorem on pointwise convergence of Fourier series came into play. For a brief while in the development of the subject, it seemed that the two topics (sharp weighted inequalities and time--frequency analysis) are coming together, but it was soon realized that the elaborate time--frequency techniques were actually superfluous for the weighted theory---at least for most of the problems considered so far. (A notable exception is the work of Do and Lacey~\cite{DoLac}, which by its very nature must lie in the intersection of the two domains.)

The $A_2$ theorem for $T_{\#}$ was recovered, sharpened and greatly simplified by Hyt\"onen and Lacey's discovery \cite{HytLac} of the alternative route 
\begin{equation}\label{eq:HL}
  \textup{corona}\ \to\ \textup{testing }S_k^+\ \to\  A_2\textup{ for }S_k^+
  \ \to\  A_2\textup{ for }S_k\ \to\  A_2\text{ for CZO};
\end{equation}
surprisingly, the full $A_2$ theorem was reduced to positive operators, a theme further elaborated in 2012.
 Before going into these most recent developments, it should be mentioned that the corona and testing condition parts can also be replaced by alternative Bellman function arguments (like those by Nazarov and Volberg \cite{NazVol}), but I would say that they remain roughly on the same level of difficulty. 

However, both corona, testing, and Bellman functions were completely avoided by Lerner's discovery \cite{Lerner:posDyad} of
``$  A_2\textup{ for }S_0^+\ \to\ 
  A_2\textup{ for }S_k^+,$''
since the starting point, $A_2$ for the simplest operator $S_0^+$, can be directly verified by an elegant few-lines argument due to Cruz-Uribe, Martell and P\'erez \cite{CUMP}.
The final shortcut
``$  A_2\textup{ for }S_k^+\ \to\
  A_2\textup{ for CZO},$''
which even avoided the random dyadic representation, was independently found by Hyt\"onen--Lacey--P\'erez \cite{HLP} and Lerner~\cite{Lerner:simpleA2}.

Altogether, it now seems that the lower route to the $A_2$ theorem,
\begin{equation}\label{eq:lower}
  A_2\textup{ for }S_0^+\quad\to\quad
  A_2\textup{ for }S_k^+\quad\to\quad
  A_2\textup{ for CZO}, 
\end{equation}
is the easiest one available as of today. On the other hand, it also seems that for a number of closely related results, it is necessary to take some additional steps. Until recently, this was the case for the $A_p$ theorem
\begin{equation}\label{eq:Ap}
  \Norm{Tf}{L^p(w)}\leq c_{T,p}([w]_{A_p}+[w]_{A_p}^{1/(p-1)})\Norm{f}{L^p(w)},\qquad 1<p<\infty,
\end{equation}
which was originally deduced from the $A_2$ theorem after an additional extrapolation argument from \cite{DGPP}. It can also be obtained directly from some paths of the above diagram by changing ``$A_2$ for $T$'' to ``$A_p$ for $T$'' and modifying the ``testing $T$'' conditions accordingly: This was achieved via the route \eqref{eq:HPTV} by Hyt\"onen, Lacey, Martikainen et al. \cite{HLMORSU}, and via \eqref{eq:HL} by Hyt\"onen and Lacey~\cite{HytLac}. However, recently Moen~\cite{Moen} found a short direct proof of ``$A_p$ for $S_0^+$'', making the easy direct route \eqref{eq:lower} also available for the full $A_p$ theorem \eqref{eq:Ap}.

Still, it seems that for the mixed $A_p$--$A_\infty$ improvement of \eqref{eq:Ap},
\begin{equation}\label{eq:ApAinfty}
  \Norm{Tf}{L^p(w)}\leq c_{T,p}[w]_{A_p}^{1/p}([w]_{A_\infty}^{1/p'}+[\sigma]_{A_\infty}^{1/p})\Norm{f}{L^p(w)},\qquad \sigma:=w^{1-p'},
\end{equation}
an approach via the testing conditions and a weighted $T1$ theorem is necessary. The bound \eqref{eq:ApAinfty} was first obtained by Hyt\"onen--P\'erez \cite{HytPer} for $p=2$ via \eqref{eq:HPTV}, and then in general by Hyt\"onen--Lacey \cite{HytLac} via \eqref{eq:HL}. As pointed out by Lerner \cite[Sec.~2.2]{Lerner:posDyad}, this can be somewhat simplified to the lower-left route
\begin{equation}\label{eq:left}
    \textup{corona}\ \to\ \textup{testing }S_0^+\ \to\  A_p\textup{ for }S_0^+
  \ \to\  A_p\textup{ for }S_k^+\ \to\  A_p\text{ for CZO},
\end{equation}
although it still needs many of the same ideas as \eqref{eq:HL} in the easier case of $k=0$.

The goal of this paper is to further simplify this lower-left route \eqref{eq:left} to the $A_p$ theorem and the mixed estimate \eqref{eq:ApAinfty}. A detailed technical outline of this route is given in Section~\ref{sec:outline}. After this, the new contributions are as follows:
\begin{itemize}
  \item[Sec.~\ref{sec:Lerner}:] A streamlined form of Lerner's local oscillation formula, where a maximal function term from the original formulation is seen to be redundant.
  \item[Sec.~\ref{sec:weak11}:] An example showing the sharpness of the known weak $(1,1)$ bound for the operators $S_k^+$, a key lemma to prove that ``$A_2$ for $S_0^+$'' implies ``$A_2$ for $S_k^+$''.
  \item[Sec.~\ref{sec:testing}:] A direct verification of the $A_p$--$A_\infty$ testing conditions for $S_0^+$. This means that the corona decomposition (and Bellman function) is once again avoided.
  \item[Sec:~\ref{sec:T1}:] A (slight) variant of the proof of the two-weight $T1$ theorem (for the simplest positive operators $S_0^+$) that is still necessary to follow this route.
\end{itemize}

The new insight might give hints towards some related open questions, of which I mention the following: For all cubes $Q$, let $X_Q,Y_Q$ by Banach function spaces on $Q$, with duals $X_Q',Y_Q'$ with respect to the duality $\fint_Q fg\ud x$, and let $M_{X'} f(x):=\sup_{Q\owns x}\Norm{f}{X_Q'}$ and $M_{Y'}$ be defined similarly. In this set-up, Lerner \cite{Lerner:posDyad} has shown that the following two-weight condition is sufficient for the boundedness of $T(\,\cdot\,\sigma):L^p(\sigma)\to L^p(w)$ for an arbitrary Calder\'on--Zygmund operator $T$: 
\begin{equation*}
  \Big(\sup_Q\Norm{w^{1/p}}{X_Q}\Norm{\sigma^{1/p'}}{Y_Q}\Big)\Norm{M_{X'}}{\mathscr{B}(L^{p'})}\Norm{M_{Y'}}{\mathscr{B}(L^p)}<\infty.
\end{equation*}
Note that the $L^{p'}$-boundedness of $M_{X'}$ roughly means that ``$X_Q'$ has a weaker norm than $L^{p'}(Q)$'', hence ``$X_Q$ has a stronger norm than $L^p(Q)$'', so that $\Norm{w^{1/p}}{X_Q}$ is a bigger (``bumped up'') quantity than $\ave{w}_{Q}^{1/p}$. Although partial progress was achieved by Cruz-Uribe, Reznikov and Volberg \cite{CRV}, it is open if the following one-sided bump condition is sufficient:
\begin{equation*}
  \Big(\sup_Q\ave{w}_Q^{1/p}\Norm{\sigma^{1/p'}}{Y_Q}\Big)\Norm{M_{Y'}}{\mathscr{B}(L^p)}+
    \Big(\sup_Q\Norm{w^{1/p}}{X_Q}\ave{\sigma}_Q^{1/p'}\Big)\Norm{M_{X'}}{\mathscr{B}(L^{p'})}<\infty.
\end{equation*}
This last quantity, for $w,\sigma\in A_\infty$ and $X_Q=L^{p+\epsilon}(Q)$, $Y_Q=L^{p'+\epsilon}(Q)$, is dominated by the product of $A_p$ and $A_\infty$ norms in \eqref{eq:ApAinfty}. Thus the one-sided bump conjecture would recover the $A_p$ theorem, while the two-sided bump theorem does not.


\section{Detailed outline of the lower-left route}\label{sec:outline}

\subsection{Lerner's formula}
The key ingredient of the recent proofs of the $A_2$ theorem is Lerner's local oscillation formula from \cite{Lerner:posDyad,Lerner:formula}. 
It involves the following concepts:
\begin{itemize}
  \item The \emph{median} of a measurable function $f$ on a set $Q$ is any real number $m_f(Q)$ such that
  \begin{equation*}
  \abs{Q\cap\{f>m_f(Q)\}}\leq\tfrac12\abs{Q},\qquad
  \abs{Q\cap\{f<m_f(Q)\}}\leq\tfrac12\abs{Q}.
\end{equation*}
  \item The \emph{decreasing rearrangement} of $f$ is the nonnegative function
  \begin{equation*}
  f^*(t):=\inf\{\alpha\geq 0:\abs{\{\abs{f}>\alpha\}}\leq t\}
    =\inf_{E:\abs{E}\leq t}\Norm{f 1_{E^c}}{\infty},
\end{equation*}
where both infima are actually reached by $\alpha=f^*(t)$ and $E=\{\abs{f}>f^*(t)\}$.
  \item The \emph{oscillation} of $f$ on $Q$, off a $\lambda$-fraction, is
  \begin{equation*}
  \omega_\lambda(f;Q):=\inf_c(1_Q(f-c))^*(\lambda\abs{Q}).
\end{equation*}
\end{itemize}

The key properties of these objects a summarized in the following simple lemma:

\begin{lemma}\label{lem:median}
We have the estimates
\begin{equation*}
\begin{split}
  &\abs{m_f(Q)} \leq(1_Q f)^*(\nu\abs{Q})\qquad  \forall\nu\in(0,\tfrac12),\qquad
  f^*(t) \leq\frac{1}{t}\Norm{f}{L^{1,\infty}}\quad\forall t\in(0,\infty),\\  
  &(1_Q(f-m_f(Q)))^*(\nu\abs{Q}) \leq 2\omega_\nu(f;Q)\qquad  \forall\nu\in(0,\tfrac12).
\end{split}
\end{equation*}
\end{lemma}

\begin{remark}
These estimates are in general invalid for $\nu=\tfrac12$ and \emph{some} medians. Indeed, consider $Q=[0,1)\subset\R^1$ and $f=1_{[0,\tfrac12)}$. Then any $c\in[0,1]$ is a median of $f$ on $Q$, but $(1_Q f)^*(\tfrac12\abs{Q})=0$, so the first bound is only true for the special median $m_f(Q)=0$. Likewise, one can check with either $c=0$ or $c=1$ that $\omega_{1/2}(f;Q)=0$, but $(1_Q(f-m_f(Q)))^*(\tfrac12\abs{Q})>0$ for all other medians $m_f(Q)\in(0,1)$.
\end{remark}

In Section~\ref{sec:Lerner} below, I prove Lerner's formula in the following form:

\begin{theorem}
For any measurable function $f$ on a cube $Q^0\subset\R^d$, we have
\begin{equation*}
  \abs{f(x)-m_f(Q^0)}\leq 2\sum_{L\in\mathscr{L}}\omega_\lambda(f;L)1_L(x),\qquad
  \lambda=2^{-d-2},
\end{equation*}
where $\mathscr{L}\subset\mathscr{D}(Q^0)$ is \emph{sparse}: there are pairwise disjoint \emph{major subsets} $E(L)\subset L$ with $\abs{E(L)}\geq\gamma\abs{L}$. In fact, we can take $\gamma=\tfrac12$.
\end{theorem}

I now discuss the application of this formula in the proof of the $A_2$ theorem.

\subsection{The reduction ``$A_2$ for $S_k^+$ $\to$ $A_2$ for CZO''}
With minor modifications, everything here extends to the maximal truncated singular integral $T_{\#}$, and even, for a smaller class of Calder\'on--Zygmund operators, a stronger nonlinearity given by the so-called $q$-variation of singular integrals $V_q^\phi T$; see Hyt\"onen--Lacey--P\'erez \cite{HLP}. But for the sake of simplicity I only present this discussion for the linear operator~$T$.

We consider a (say, bounded) compactly supported $f$ and pick some $Q^0\supset\supp f$. Lerner's formula (applied to $Tf$) guarantees that
\begin{equation*}
  \abs{Tf(x)}\leq \abs{m_{Tf}(Q_0)}+2\sum_{L\in\mathscr{L}}\omega_\lambda(Tf;L)1_L(x).
\end{equation*}
Using the first and third estimates from Lemma~\ref{lem:median}, we see that
\begin{equation*}
  \abs{m_{Tf}(Q_0)}\leq (1_{Q_0} Tf)^*(\nu\abs{Q_0})
  \leq\frac{1}{\nu\abs{Q_0}}\Norm{Tf}{L^{1,\infty}}
  \lesssim\frac{1}{\abs{Q_0}}\Norm{f}{L^1}=\fint_{Q_0}\abs{f}
\end{equation*}
by the boundedness $T:L^1\to L^{1,\infty}$, and the fact that $\supp f\subseteq Q_0$ in the last two steps.
We also have the following estimate that essentially goes back to Jawerth and Torchinsky \cite{JawTor}:

\begin{lemma}
Let the kernel $K$ of $T$ satisfy
\begin{equation*}
  \abs{K(x,y)-K(x',y)}\lesssim\Omega\Big(\frac{\abs{x-x'}}{\abs{x-y}}\Big)\frac{1}{\abs{x-y}}\qquad
  \forall\ \abs{x-y}>2\abs{x-x'}>0,
\end{equation*}
where the modulus of continuity $\Omega:[0,\infty)\to[0,\infty)$ is increasing and subadditive with $\Omega(0)=0$. Then\quad
$\displaystyle
  \omega_\lambda(Tf;Q)\lesssim\sum_{k=0}^\infty\Omega(2^{-k})\fint_{2^k Q}\abs{f}.
$
\end{lemma}

To replace the concentric expansion $2^k Q$ by the dyadic ancestor $Q^{(k)}$, we use the following geometric lemma, well known for $k=0$, and proven in \cite{HLP} as stated here.

\begin{lemma}
For any cube $Q\subset\R^d$, there exists a shifted dyadic cube
\begin{equation*}
  R\in \mathscr{D}^\alpha:=\{2^{-k}([0,1)^d+m+(-1)^k\alpha): k\in\Z,m\in\Z^d\},
\end{equation*}
for some $\alpha\in\{0,\tfrac13,\tfrac23\}^d$, such that
$\quad
  Q\subseteq R,\quad 2^k Q\subseteq R^{(k)},\quad \ell(R)\leq 6\ell(Q).
$
\end{lemma}

Let us denote the index $\alpha$ and cube $R$ produced by this lemma by $\alpha(Q,k)$ and $R(Q,k)$.
%
Thus we have
\begin{equation}\label{eq:1stRed}
\begin{split}
  &\abs{Tf(x)} \cdot 1_{Q_0}(x)
  \lesssim \fint_{Q_0}\abs{f}\cdot 1_{Q_0}(x)+\sum_{L\in\mathscr{L}}\omega_\lambda(Tf;L)1_L(x) \\
    &\lesssim \fint_{Q_0}\abs{f}\cdot 1_{Q_0}(x)+\sum_{L\in\mathscr{L}}\sum_{k=0}^\infty\Omega(2^{-k})\fint_{2^k L}\abs{f}\cdot 1_L(x) \\
  &\lesssim \sum_\alpha\sum_{k=0}^\infty \Omega(2^{-k})\sum_{R\in\mathscr{R}^\alpha_k}\fint_{R^{(k)}}\abs{f}\cdot 1_R(x) 
  =:\sum_\alpha\sum_{k=0}^\infty \Omega(2^{-k}) S_{\alpha,k}^+(\abs{f})(x)
\end{split}
\end{equation}
where
\begin{equation*}
  \mathscr{R}^\alpha_k:=\{R(L,k):L\in\mathscr{D},\alpha(L,k)=\alpha\}\subset\mathscr{D}^\alpha
\end{equation*}
and $1_{Q_0}(x)\fint_{Q_0}\abs{f}$ was absorbed into the sum $\sum_{R\in\mathscr{R}^\alpha_k}$ with $\alpha=k=0$.

The collections $\mathscr{R}^\alpha_k$ are sparse; indeed, the sets $E(L)$, $L\in\mathscr{L}$, are pairwise disjoint, and
\begin{equation*}
  \abs{E(L)}\geq 2^{-1}\abs{L}\geq 2^{-1}6^{-d}\abs{R(L,k)}.
\end{equation*}
%

Since $Q_0\supset\supp f$ was arbitrary, from \eqref{eq:1stRed} we see that to estimate $Tf$ in a Banach function space, it clearly suffices to estimate the $S^+_{\alpha,k}(\abs{f})$ in the same space, and this proves the claimed reduction.

\subsection{The reduction ``$A_2$ for $S^+_0$ $\to$ $A_2$ for $S^+_k$''}
This reduction is due to Lerner. Since we deal with positive operators, we can also restrict to nonnegative functions.
 Dualizing $S^+_k f$ with a function $g$ of bounded support (choosing $Q_0\supset\supp f,\supp g$), we have
\begin{equation*}
   \pair{S^+_{\alpha,k}f}{g} =\pair{f}{(S^+_{\alpha,k})^t g},
\end{equation*}
and we apply Lerner's formula to $(S^+_{\alpha,k})^t g$ on a large enough cube $Q_\alpha\in\mathscr{D}^\alpha$. This gives
\begin{equation*}
  (S^+_{\alpha,k})^t g(x)
  \leq m_{(S^+_{\alpha,k})^t g}(Q_\alpha)
  +\sum_{L\in\mathscr{L}_{\alpha,k}}\omega_\lambda((S^+_{k,\alpha})^t g,L)1_L(x),
\end{equation*}
and one can check that
\begin{equation*}
  \omega_\lambda((S^+_{k,\alpha})^t g;L)
  \lesssim(1+k)\fint_L g,
\end{equation*}
where the essential ingredient is the weak $(1,1)$ bound for dyadic shifts, the only remaining component of the original proof of the $A_2$ conjecture \cite[Prop.~5.1]{Hytonen:A2}. This holds for a general dyadic shift of complexity $k$, and it is not in any substantial way easier for the particular shifts $(S_{k,\alpha}^+)^*$, a special case treated in \cite[Lemma 3.2]{Lerner:simpleA2}.

\begin{proposition}\label{prop:weak11}
Any dyadic shift $S_k$ of order $k$ satisfies
\begin{equation*}
  \Norm{S_k f}{L^{1,\infty}}\lesssim(1+k)\Norm{f}{L^1}.
\end{equation*}
\end{proposition}

The median term is handled similarly, and the result is that
\begin{equation*}
  (S^+_{\alpha,k})^t g(x)
  \lesssim (1+k)\sum_{L\in\mathscr{L}_{\alpha,k}} \fint_Lg\cdot 1_L(x)
  =:(1+k)(S^+_{0,(\alpha,k)})g(x),
\end{equation*}
where $S^+_{0,(\alpha,k)}$ is a positive dyadic shift of complexity zero, a self-dual operator. Hence
\begin{equation*}
  \pair{S^+_{\alpha,k} f}{g}
  =\pair{f}{(S^+_{\alpha,k})^t g}
  \leq(1+k)\pair{S^+_{0,(\alpha,k)}f}{g}.
\end{equation*}

If we combine this with the previous reduction, we arrive at
\begin{equation*}
  \abs{\pair{Tf}{g}}
  \lesssim\sum_\alpha\sum_{k=0}^\infty\Omega(2^{-k})(1+k)\pair{S^+_{0,(\alpha,k)}(\abs{f})}{\abs{g}},
\end{equation*}
and thus, for any Banach function space $X$,
\begin{equation*}
  \Norm{Tf}{X}
  \lesssim\Big(\sum_{k=0}^\infty\Omega(2^{-k})(1+k)\Big)\sup_{S^+_0}\Norm{S^+_0\abs{f}}{X},
\end{equation*}
where the supremum is over all positive dyadic shifts of complexity zero. As long as the above series converges (which is clear for the H\"older moduli $\Omega(t)=t^\delta$, but also for some weaker moduli of continuity), if suffices to estimate the norm of $S^+_0\abs{f}$.

Note that $\sum_{k=0}^\infty\Omega(2^{-k})<\infty$ is the classical Dini condition $\int_0^1\Omega(t)\ud t/t<\infty$, where as we require a logarithmic strengthening $\int_0^1\Omega(t)(1+\log(1/t))\ud t/t<\infty$ caused by Proposition~\ref{prop:weak11}. While it is open if this logarithm is actually necessary in the final result, any attempt to remove it would have to circumvent the application of Proposition~\ref{prop:weak11}. Namely, this proposition is actually sharp, already for the special shifts $S^+_{\alpha,k}$, as shown by example in Section~\ref{sec:weak11}.

%
%

\subsection{The reduction ``testing $S^+_0$ $\to$ $A_2$ for $S^+_0$''}
This is a direct application of the following elegant two-weight result of Lacey, Sawyer and Uriarte-Tuero \cite{LSU:positive}. A non-dyadic variant of this result goes back to Sawyer \cite{Sawyer:2wFractional}, and the stated dyadic version in the case $p=q=2$ to Nazarov--Treil--Volberg \cite{NTV:bilinear}. There is also a more recent simplification of the proof of the full theorem by Treil \cite{Treil:posDyad}. 

\begin{theorem}\label{thm:LSU}
Let $\lambda_Q\geq 0$ be some coefficients, and
\begin{equation}\label{eq:posDyad}
  S f:=\sum_{Q\in\mathscr{D}}\lambda_Q\ave{f}_Q 1_Q,\qquad
  S_Q f:=\sum_{\substack{Q'\in\mathscr{D}\\ Q'\subseteq Q}}\lambda_{Q'}\ave{f}_{Q'} 1_{Q'}
\end{equation}
be the associated positive dyadic shift of complexity zero, and its subshifts. For any two weights $w$ and $\sigma$, we have
\begin{equation*}
  \sup_f\frac{\Norm{S(f\sigma)}{L^p(w)}}{\Norm{f}{L^p(\sigma)}}
  \eqsim\sup_{Q\in\mathscr{D}}\frac{\Norm{S_Q(\sigma)}{L^p(w)}}{\sigma(Q)^{1/p}}
     +\sup_{Q\in\mathscr{D}}\frac{\Norm{S_Q(w)}{L^{p'}(\sigma)}}{w(Q)^{1/p'}}.
\end{equation*}
\end{theorem}

Observe that the shifts $S_0^+$ correspond to the special case where $\lambda_Q = 1_{\mathscr{L}}(Q)$ for some sparse family $\mathscr{L}\subset\mathscr{D}$. Note also that with a special choice of the other weight, we have
\begin{equation*}
  \sup_f{\Norm{S(f\sigma)}{L^p(w)}}\big/{\Norm{f}{L^p(\sigma)}}
  =\sup_f{\Norm{Sf}{L^p(w)}}\big/{\Norm{f}{L^p(w)}},\qquad
  \sigma=w^{1-p'}.
\end{equation*}

Thus for example the mixed $A_p$--$A_\infty$ bound
\begin{equation*}
  \Norm{S_0^+f}{L^p(w)}
  \lesssim[w]_{A_p}^{1/p}([w]_{A_\infty}^{1/p'}+[w^{1-p'}]_{A_\infty}^{1/p})\Norm{f}{L^p(w)}
\end{equation*}
is a special case of the two-weight bound
\begin{equation*}
  \Norm{S_0^+ (f\sigma)}{L^p(w)}
  \lesssim[w,\sigma]_{A_p}^{1/p}([w]_{A_\infty}^{1/p'}+[w^{1-p'}]_{A_\infty}^{1/p})\Norm{f}{L^p(w)},
\end{equation*}
where
\begin{equation}\label{eq:2wAp}
  [w,\sigma]_{A_p}:=\sup_Q\ave{w}_Q\ave{\sigma}_Q^{p-1},
\end{equation}
which in turn follows, by Theorem~\ref{thm:LSU} and symmetry, from the testing condition
\begin{equation}\label{eq:testToProve}
  \Norm{S_0^+ (1_Q\sigma)}{L^p(w)}^p \lesssim[w,\sigma]_{A_p}[\sigma]_{A_\infty}\sigma(Q).
\end{equation}
In Section~\ref{sec:testing}, I give a new direct proof of the bound \eqref{eq:testToProve}, without using either a corona decomposition or a Bellman function technique.

\section{Proof of Lerner's formula}\label{sec:Lerner}

%
%
%
%
For any family of pairwise disjoint subcubes $Q^1_j$ of $Q^0$, we can write the median Calder\'on--Zygmund decomposition
\begin{equation}\label{eq:LernerStep1}
\begin{split}
  1_{Q^0}(f-m_f(Q^0)) 
  &=1_{Q^0\setminus\bigcup Q^1_j}(f-m_f(Q^0))  +\sum_j 1_{Q^1_j}(m_f(Q^1_j)-m_f(Q^0)) \\
  &\qquad+\sum_j 1_{Q^1_j}(f-m_f(Q^1_j)).
\end{split}
\end{equation}
We apply this with the following specific choice of the stopping cubes $Q^1_j$: they are the maximal dyadic subcubes of $Q^0$ with the property that
\begin{equation}\label{eq:stoppingCond}
  \max_{Q'\in\textup{ch}(Q^1_j)}\abs{m_f(Q')-m_f(Q^0)}>(1_{Q^0}(f-m_f(Q^0)))^*(\lambda\abs{Q^0}),
\end{equation}
where $\textup{ch}(Q):=\{Q'\in\mathscr{D}(Q):\ell(Q')=\tfrac\ell(Q)\}$ is the collection of dyadic children of $Q$.
From the maximality it follows that $Q^1_j$ in place of $Q'\in\textup{ch}(Q^1_j)$ satisfies the opposite estimate, and hence the second term on the right of \eqref{eq:LernerStep1} is dominated by
\begin{equation*}
 1_{\bigcup Q^1_j}(1_{Q^0}(f-m_f(Q^0)))^*(\lambda\abs{Q^0})
 \leq 1_{\bigcup Q^1_j}\cdot 2\omega_\lambda(f;Q^0).
\end{equation*}
On the other hand, if $x\in Q^0\setminus\bigcup Q^1_j$, then the estimate opposite to \eqref{eq:stoppingCond} holds for all dyadic $Q'\owns x$. A lemma of Fujii \cite[Lemma 2.2]{Fujii} (``a Lebesgue differentiation theorem for the median'') guarantees that $m_f(Q')\to f(x)$ as $Q'\to x$ for almost every $x$, and hence also the first term on the right of \eqref{eq:LernerStep1} is dominated by
\begin{equation*}
 1_{Q^0\setminus \bigcup Q^1_j}(1_{Q^0}(f-m_f(Q^0)))^*(\lambda\abs{Q^0})
 \leq 1_{Q^0\setminus\bigcup Q^1_j}\cdot 2\omega_\lambda(f;Q^0).
\end{equation*}
Altogether, we find that
\begin{equation*}
  \abs{1_{Q^0}(f-m_f(Q^0))}
  \leq 1_{Q^0}\cdot 2\omega_\lambda(f;Q^0)
  +\sum_j\abs{1_{Q^1_j}(f-m_f(Q^1_j))},
\end{equation*}
where the terms in the sum are of the same form as the left side, with $Q^0$ replaced by $Q^1_j$, and we are in a position to iterate. This gives
\begin{equation}\label{eq:iteration}
\begin{split}
  &\abs{1_{Q^0}(f-m_f(Q^0))} \\
   &\leq 1_{Q^0}\cdot 2\omega_\lambda(f;Q^0)
     +\sum_j 1_{Q^1_j}\cdot 2\omega_\lambda(f;Q^1_j) +\sum_i \abs{1_{Q^2_i}(f-m_f(Q^2_i))}  \\
   &\leq\ldots
   \leq\sum_{k=0}^{m}\sum_j
     1_{Q^k_j}\cdot 2\omega_\lambda(f;Q^k_j) +\sum_i \abs{1_{Q^{m+1}_i}(f-m_f(Q^{m+1}_i))},
\end{split}
\end{equation}
where the cubes $Q^{m+1}_i$ are dyadic subcubes of some $Q^{m}_j$, chosen by a similar stopping criterion as the $Q^1_j$ from $Q^0$ in \eqref{eq:stoppingCond}.

We claim that
\begin{equation}\label{eq:sparseness}
  \abs{E(Q^m_j)}:=\Babs{Q^m_j\setminus\bigcup_i Q^{m+1}_i}\geq\frac{1}{2}\abs{Q^m_j}.
\end{equation}
This would show in particular that
\begin{equation*}
 \abs{\Omega^{m+1}}:=\Babs{\bigcup_i Q^{m+1}_i}\leq 2^{-1}\abs{\Omega^m}\leq\ldots\leq 2^{-m-1}\abs{Q^0},
\end{equation*}
and hence the last term in \eqref{eq:iteration} is supported on a set $\Omega^{m+1}\subset\Omega^m\subset\ldots\subset Q^0$ of measure at most $2^{-m-1}\abs{Q^0}$. As $m\to\infty$, the support of this last term tends to a null set, and hence we obtain that
\begin{equation*}
  \abs{1_{Q^0}(f-m_f(Q^0))}
  \leq \sum_{k=0}^{\infty}\sum_j
     1_{Q^k_j}\cdot 2\omega_\lambda(f;Q^k_j) 
\end{equation*}
pointwise almost everywhere. This is the claimed formula with $\mathscr{L}:=\{Q^k_j\}_{k,j}$, and it only remains to check the sparseness condition \eqref{eq:sparseness}. By symmetry, it suffices to consider $m=0$.

We abbreviate $f_0:=f-m_f(Q^0)$. Then the stopping condition gives for some $Q'\in\textup{ch}(Q^1_j)$ and any $\nu\in(0,\tfrac12)$ the estimate
\begin{equation*}
  \alpha:=(1_{Q^0}f_0)^*(\lambda\abs{Q^0})
  <\abs{m_{f_0}(Q')}\leq(1_{Q'}f_0)^*(\nu\abs{Q'})\leq (1_{Q^1_j}f_0)^*(\nu 2^{-d}\abs{Q^1_j}).
\end{equation*}
Thus
$  \abs{Q^1_j\cap\{\abs{f_0}>\alpha\}}\geq\nu 2^{-d}\abs{Q^1_j},
$
and hence
\begin{equation*}
  \nu 2^{-d}\sum_j\abs{Q^1_j}
  \leq\sum_j\abs{Q^1_j\cap\{\abs{f_0}>\alpha\}} \\
  \leq\abs{Q^0\cap\{\abs{f_0}>\alpha\}}
  \leq\lambda\abs{Q^0}=2^{-d-2}\abs{Q^0}.
\end{equation*}
Letting $\nu\to\tfrac12$, we get
$\displaystyle
  \sum_j\abs{Q^1_j}\leq\tfrac12\abs{Q^0},
$
which is the same as \eqref{eq:sparseness} for $m=0$.

\section{Sharpness of the weak $(1,1)$ estimate}\label{sec:weak11}

I show by example (on $\R^1$) that the known bound
\begin{equation*}
  \Norm{(S_k^+)^*f}{L^{1,\infty}}\lesssim (1+k)\Norm{f}{L^1}
\end{equation*}
is sharp in terms of dependence on $k$.

For $L\in\mathscr{D}$, let $L_{(j)}$ be the dyadic interval with $(L_{(j)})^{(j)}=L$ and $\inf L_{(j)}=\inf L$.

\begin{equation*}
  \mathscr{L}:=\{L\in\mathscr{D}:L^{(k)}=[0,1)\},\qquad
  \mathscr{K}:=\bigcup_{L\in\mathscr{L}}\{L_{(j)}:j=0,\ldots,k\}.
\end{equation*}
Clearly this is a sparse family, with
\begin{equation*}
  \abs{E(K)}:=\Babs{K\setminus\bigcup_{\substack{K'\in\mathscr{K}\\ K'\subsetneq K}}K'}\geq\frac12\abs{K}
 \qquad\forall K\in\mathscr{K}.
\end{equation*}
Consider $f:=2^k \sum_{L\in\mathscr{L}}  1_{L_{(k)}}$
so that $\Norm{f}{L^1}=2^k\sum_{L\in\mathscr{L}}\abs{L_{(k)}}=\sum_{L\in\mathscr{L}}\abs{L}=1$. For $K=L_{(j)}\in\mathscr{K}$, we then have $\int_K f=\abs{L}$.
Note that the intervals $(L_{(j)})^{(k)}=L^{(k-j)}$, $L\in\mathscr{L}$ cover every point of $[0,1)$ exactly $2^{k-j}$ times. Hence
\begin{equation*}
\begin{split}
  \sum_{K\in\mathscr{K}}\frac{1_{K^{(k)}}}{\abs{K^{(k)}}}\int_K f
  &=\sum_{j=0}^k\sum_{L\in\mathscr{L}}\frac{1_{(L_{(j)})^{(k)}}}{\abs{(L_{(j)})^{(k)}}}\int_{L_{(j)}}f
  =\sum_{j=0}^k\sum_{L\in\mathscr{L}}\frac{1_{L^{(k-j)}}}{\abs{L^{(k-j)}}}\abs{L} \\
  & =\sum_{j=0}^k 2^{j-k}\sum_{L\in\mathscr{L}}1_{L^{(k-j)}} 
     =\sum_{j=0}^k 2^{j-k} \times 2^{k-j}1_{[0,1)}
     =(k+1)1_{[0,1)}.
\end{split}
\end{equation*}
Then clearly
 $ \Norm{(S_k^+)^*f}{L^{1,\infty}}=\Norm{(k+1)1_{[0,1)}}{L^{1,\infty}}=k+1=(k+1)\Norm{f}{L^1}.$

\section{A direct verification of the testing conditions}\label{sec:testing}

We want to prove that
\begin{equation*}
  \Norm{S_Q\sigma}{L^p(w)}^p
  =\int_Q\Big(\sum_{\substack{L\in\mathscr{L} \\ L\subseteq Q}}\ave{\sigma}_L 1_L\Big)^p w
  \lesssim[w,\sigma]_{A_p}[\sigma]_{A_\infty}\sigma(Q),
\end{equation*}
where $L$ is a sparse family of cubes. Henceforth, we will suppress the summation condition ``$L\in\mathscr{L}$'' with the understanding that all summation variables $L,L',L_1,L_2,\ldots$ are always taken from the collection $\mathscr{L}$. Recall that the two-weight $A_p$ constant is defined by \eqref{eq:2wAp}, and the $A_\infty$ constant by
\begin{equation*}
  [\sigma]_{A_\infty}:=\sup_Q\frac{1}{\sigma(Q)}\int_Q M(1_Q\sigma).
\end{equation*}

\subsection{The $A_2$ case}
This case is particularly simple:
\begin{equation*}
\begin{split}
  \int_Q &\Big(\sum_{L\subseteq Q}\ave{\sigma}_L 1_L\Big)^2 w
  \leq 2\int_Q\sum_{L\subseteq Q}\sum_{L'\subseteq L}\ave{\sigma}_L\ave{\sigma}_{L'}1_{L'}w \\
  &=2\sum_{L\subseteq Q}\ave{\sigma}_L\sum_{L'\subseteq L}\ave{\sigma}_{L'}\ave{w}_{L'}\abs{L'} 
  \leq 2[w,\sigma]_{A_2}\sum_{L\subseteq Q}\ave{\sigma}_L\sum_{L'\subseteq L}\abs{L'} \\
  &\lesssim [w,\sigma]_{A_2}\sum_{L\subseteq Q}\ave{\sigma}_L\abs{L} 
  \lesssim [w,\sigma]_{A_2}\sum_{L\subseteq Q}\inf_{L}M(\sigma 1_Q)\cdot\abs{E(L)} \\
  &\leq [w,\sigma]_{A_2}\sum_{L\subseteq Q}\int_{E(L)}M(\sigma 1_Q)
  \leq [w,\sigma]_{A_2}\int_Q M(\sigma 1_Q)  
  \leq [w,\sigma]_{A_2}[\sigma]_{A_\infty}\sigma(Q).
\end{split}
\end{equation*}

\subsection{The general $A_p$ case}
To ``multiply out'' the expression
\begin{equation*}
  \Big(\sum_{L\subseteq Q}\ave{\sigma}_L 1_L\Big)^p
\end{equation*}
for a possibly non-integer value of $p\in(1,\infty)$, we need the following observation:

\begin{lemma}\label{lem:multOut}
For all $k\in\mathbb{N}$ and $\alpha\in[0,1]$, and all nonnegative sequences of numbers $a_i$, we have
\begin{equation*}
\begin{split}
  \Big(\sum_i a_i\Big)^{k+\alpha}
  &\leq (k+1)\sum_{i_1,\ldots,i_k}a_{i_1}\cdots a_{i_k}\Big(\sum_{j\leq\min\{i_1,\ldots,i_k\}}a_j\Big)^{\alpha} \\
  &\leq(k+1)!\sum_{i_1\geq i_2\geq\ldots\geq i_k\geq j}a_{i_1}\cdots a_{i_k}\cdot a_{j}^{\alpha}.
\end{split}
\end{equation*}
\end{lemma}

\begin{proof}
Note that the second estimate is obvious, we only prove the first one.
Let $A_i:=\sum_{j\leq i}a_j$, and for $\vec{i}:=(i_1,\ldots,i_k)$, write $a_{\vec{i}}:=a_{i_1}\cdots a_{i_k}$ and $A_{\vec{i}}:=A_{\min\{i_1,\ldots,i_k\}}$. Then the claim reads as $A_{\infty}^{k+\alpha}\leq(k+1)\sum_{\vec{i}}a_{\vec{i}}A_{\vec{i}}^{\alpha}$. We consider the fraction
\begin{equation*}
  f(\alpha):=A_{\infty}^{k+\alpha}\Big/\sum_{\vec{i}}a_{\vec{i}}A_{\vec{i}}^{\alpha},
\end{equation*}
and prove that it is bounded by $(k+1)$.
Its derivative satisfies
\begin{equation*}
  f'(\alpha)=\frac{A_{\infty}^{k+\alpha}}{\Big(\sum_{\vec{i}}a_{\vec{i}}A_{\vec{i}}^{\alpha}\Big)^2}
    \sum_{\vec{i}}(\log A_{\infty}-\log A_{\vec{i}})a_{\vec{i}} A_{\vec{i}}^{\alpha}\geq 0,
\end{equation*}
since $A_{\infty}\geq A_{\vec{i}}$, and all quantities are nonnegative. Thus $f(\alpha)\leq f(1)$, and it is clear that this is at most $k+1$, since $A_{\infty}^{k+1}=\sum_{i_1,\ldots,i_{k+1}}a_{i_1}\cdots a_{i_{k+1}}$, and there are $k+1$ possible choices for which of the indices $i_1,\ldots,i_{k+1}$ is the smallest.
\end{proof}

For $1\leq k<p\leq k+1$, Lemma~\ref{lem:multOut} gives
\begin{equation*}
\begin{split}
  \int_Q\Big(\sum_{L\subseteq Q}\ave{\sigma}_L 1_L\Big)^p w
  &\lesssim \int_Q\sum_{Q\supseteq L_1\supseteq\ldots\supseteq L_k\supseteq L_{k+1}}\ave{\sigma}_{L_1}
    \cdots\ave{\sigma}_{L_k}\ave{\sigma}_{L_{k+1}}^{p-k} 1_{L_{k+1}}w \\
  &=\sum_{Q\supseteq L_1\supseteq\ldots\supseteq L_k\supseteq L_{k+1}}\ave{\sigma}_{L_1}
    \cdots\ave{\sigma}_{L_k}\ave{\sigma}_{L_{k+1}}^{p-k} \ave{w}_{L_{k+1}}\abs{L_{k+1}} \\
\end{split}
\end{equation*}

To proceed more smoothly, we record two further lemmas:

\begin{lemma}\label{lem:max}
For $\gamma\in[0,1)$, we have
 $\displaystyle
  \sum_{L:L\subseteq P}\ave{w}_L^\gamma\abs{L}
  \lesssim\ave{w}_P^\gamma\abs{P}.
$
\end{lemma}

\begin{proof}
\begin{equation*}
\begin{split}
   \sum_{L:L\subseteq P}\ave{w}_L^\gamma\abs{L}
   &\lesssim \sum_{L:L\subseteq P}\ave{w}_L^\gamma\abs{E(L)}
   \leq \sum_{L:L\subseteq P}\inf_L M(w1_P)^\gamma\cdot\abs{E(L)}
   \leq \int_P M(w1_P)^\gamma \\
   &\lesssim\Norm{M(w 1_P)}{L^{1,\infty}}^\gamma\cdot\abs{P}^{1-\gamma}
   \leq\Norm{w 1_P}{L^1}^\gamma\cdot\abs{P}^{1-\gamma}=\ave{w}_P^\gamma\cdot\abs{P}.\qedhere
\end{split}
\end{equation*}
\end{proof}

\begin{lemma}\label{lem:sum}
For all $0\leq\alpha\leq\beta(p-1)<\alpha+p-1$, we have
\begin{equation*}
  \sum_{L:L\subseteq P}\ave{\sigma}_L^{\alpha} \ave{w}_L^{\beta}\abs{L}
  \lesssim[w,\sigma]_{A_p}^{\alpha/(p-1)}\ave{w}_P^{\beta-\alpha/(p-1)}\abs{P}.
\end{equation*}
\end{lemma}

\begin{proof}
\begin{equation*}
\begin{split}
   \sum_{L:L\subseteq P} &\ave{\sigma}_L^\alpha \ave{w}_L^{\beta}\abs{L}
   =\sum_{L:L\subseteq P}\big(\ave{\sigma}_L^{p-1}\ave{w}_L\big)^{\alpha/(p-1)}\ave{w}_L^{\beta-\alpha/(p-1)}\abs{L} \\
   &\leq [w,\sigma]_{A_p}^{\alpha/(p-1)}\sum_{L:L\subseteq P}\ave{w}_L^{\beta-\alpha/(p-1)}\abs{L} 
   \lesssim [w,\sigma]_{A_p}^{\alpha/(p-1)}\ave{w}_P^{\beta-\alpha/(p-1)}\abs{P},
\end{split}
\end{equation*}
where we used the assumption that $\beta-\alpha/(p-1)\in[0,1)$ and Lemma~\ref{lem:max}.
\end{proof}

With $\alpha=p-k$, $\beta=1$, Lemma~\ref{lem:sum} gives
\begin{equation*}
\begin{split}
  &\sum_{Q\supseteq L_1\supseteq\ldots\supseteq L_k\supseteq L_{k+1}}\ave{\sigma}_{L_1}
    \cdots\ave{\sigma}_{L_k}\ave{\sigma}_{L_{k+1}}^{p-k} \ave{w}_{L_{k+1}}\abs{L_{k+1}} \\
   &\lesssim [w,\sigma]_{A_p}^{(p-k)/(p-1)}   
   \sum_{Q\supseteq L_1\supseteq\ldots\supseteq L_k}\ave{\sigma}_{L_1}
    \cdots\ave{\sigma}_{L_k}\ave{w}_{L_{k}}^{(k-1)/(p-1)}\abs{L_{k}} \\
\end{split}
\end{equation*}
Then, using Lemma~\ref{lem:sum} subsequently with $\alpha=1$ and $\beta=j/(p-1)$, where $j=k-1,\ldots,1$, we obtain
\begin{equation*}
\begin{split}
   &\sum_{Q\supseteq L_1\supseteq\ldots\supseteq L_k}\ave{\sigma}_{L_1}
    \cdots\ave{\sigma}_{L_k}\ave{w}_{L_{k}}^{(k-1)/(p-1)}\abs{L_{k}} \\
   &\lesssim[w,\sigma]_{A_p}^{1/(p-1)}\sum_{Q\supseteq L_1\supseteq\ldots\supseteq L_{k-1}}\ave{\sigma}_{L_1}
    \cdots\ave{\sigma}_{L_{k-1}}\ave{w}_{L_{k-1}}^{(k-2)/(p-1)}\abs{L_{k-1}} \\
   &\lesssim\ldots\lesssim[w,\sigma]_{A_p}^{(k-1)/(p-1)}\sum_{Q\supseteq L_1}\ave{\sigma}_{L_1}\abs{L_1},
\end{split}
\end{equation*}
and here
\begin{equation*}
  \sum_{L_1\subseteq Q}\ave{\sigma}_{L_1}\abs{L_1}
  \lesssim\sum_{L_1\subseteq Q}\inf_{L_1}M(\sigma 1_Q)\cdot\abs{E(L_1)}
  \leq \int_Q M(\sigma 1_Q)\leq [\sigma]_{A_\infty}\sigma(Q).
\end{equation*}
Thus, altogether, we have checked that
\begin{equation*}
  \int_Q\Big(\sum_{L\subseteq Q}\ave{\sigma}_L 1_L\Big)^p w
  \lesssim[w,\sigma]_{A_p}^{(p-k)/(p-1)}[w,\sigma]_{A_p}^{(k-1)/(p-1)}[\sigma]_{A_\infty}\sigma(Q),
\end{equation*}
and the total power of $[w,\sigma]_{A_p}$ is one, as claimed.

\section{The two-weight $T1$ theorem for positive operators}\label{sec:T1}

To be in line with the $T1$ literature, I now write $T$ and $T_Q$ instead of $S$ and $S_Q$ as defined in \eqref{eq:posDyad}. I restate and prove the main estimate of Theorem~\ref{thm:LSU} in the following form:
%

\begin{theorem}\label{thm:LSU2}
For $1<p\leq q<\infty$ and any two weights $\sigma,\omega$, we have
\begin{equation*}
\begin{split}
  \Norm{T(\,\cdot\,\sigma)}{L^p(\sigma)\to L^q(\omega)}
  &\leq 20(p'q\mathfrak{T}+pq'\mathfrak{T}^*),\\
  \mathfrak{T}:=\sup_{Q\in\mathscr{D}}\frac{\Norm{T_Q(\sigma)}{L^q(\omega)}}{\sigma(Q)^{1/p}},&\qquad
  \mathfrak{T}^*:=\sup_{Q\in\mathscr{D}}\frac{\Norm{T_Q(\omega)}{L^{p'}(\sigma)}}{\omega(Q)^{1/q'}}.
\end{split}
\end{equation*}
\end{theorem}

The proof below follows the main lines of Treil's argument~\cite{Treil:posDyad}, with one key difference: rather than splitting the summation over $Q\in\mathscr{D}$ in the expansion of $\pair{T(f\sigma)}{g\omega}$ into parts according to an \emph{ad hoc} criterion such as $\sigma(Q)\ave{f}_Q^p\geq\omega(Q)\ave{g}_Q^{p'}$ , I simply apply the ``parallel corona'' decomposition from the recent work of Lacey, Sawyer, Shen and Uriarte-Tuero \cite{LSSU} on the two-weight boundedness of the Hilbert transform. Thus, the proof below can also been seen as a toy introduction to some of the innovative techniques of~\cite{LSSU}.

\begin{proof}
We analyse the pairing
\begin{equation}\label{eq:Tfg}
  \pair{T(f\sigma)}{g\omega}
  =\sum_{Q\in\mathscr{D}}\lambda_Q\ave{f\sigma}_Q\ave{g\omega}_Q\abs{Q}
  =\sum_{Q\in\mathscr{D}}\lambda_Q\ave{f}_Q^\sigma\ave{\sigma}_Q\ave{g}_Q^\omega\ave{\omega}_Q\abs{Q}.
\end{equation}
It suffices to make a uniform estimate over all subseries with $Q\subseteq Q_0$ for some large dyadic cube $Q_0$, and
we may assume that both $f,g\geq 0$ are supported in $Q_0$. Then we define the collections of principal cubes $\mathscr{F}$ for $(f,\sigma)$ and $\mathscr{G}$ for $(g,\omega)$. Namely,
\begin{equation*}
\begin{split}
  \mathscr{F} &:=\bigcup_{k=0}^\infty\mathscr{F}_k,\qquad\textup{where}\qquad\mathscr{F}_0:=\{Q_0\},\\
  \mathscr{F}_{k+1} &:=\bigcup_{F\in\mathscr{F}_k}\textup{ch}_{\mathscr{F}}(F),\qquad
  \textup{ch}_{\mathscr{F}}(F):=\{Q\subsetneq F\textup{ maximal s.t. }\ave{f}_Q^\sigma>2\ave{f}_F^\sigma\},
\end{split}
\end{equation*}
and analogously for $\mathscr{G}$. Observe that
\begin{equation*}
  \sum_{F'\in\textup{ch}_{\mathscr{F}}(F)}\sigma(F')
  \leq\sum_{F'\in\textup{ch}_{\mathscr{F}}(F)}\frac{\int_{F'} f\sigma}{2\ave{f}_F^\sigma}
  \leq\frac{\int_{F} f\sigma}{2\ave{f}_F^\sigma}=\frac{\sigma(F)}{2},
\end{equation*}
and hence
\begin{equation*}
  \sigma(E_\mathscr{F}(F))
  :=\sigma\Big(F\setminus\bigcup_{F'\in\textup{ch}_{\mathscr{F}}(F)}F'\Big)\geq\frac{1}{2}\sigma(F),
\end{equation*}
where the sets $E_{\mathscr{F}}(F)$ are pairwise disjoint.

We further define the stopping parents
\begin{equation*}
  \pi_{\mathscr{F}}(Q):=\min\{F\supseteq Q:F\in\mathscr{F}\},\qquad
  \pi(Q):=\big(\pi_{\mathscr{F}}(Q),\pi_{\mathscr{G}}(Q)\big).
\end{equation*}
Then we rearrange the series in \eqref{eq:Tfg} as
\begin{equation}\label{eq:TfgSplit}
  \sum_Q=\sum_{\substack{F\in\mathscr{F}\\ G\in\mathscr{G}}}\sum_{\substack{Q:\\ \pi(Q)=(F,G)}}
    =\sum_F\sum_{G\subseteq F}\sum_{\substack{Q:\\ \pi(Q)=(F,G)}}+\sum_F\sum_{F\subsetneq G}\sum_{\substack{Q:\\ \pi(Q)=(F,G)}},
\end{equation}
where we observed that if the inner sum over $Q:\pi(Q)=(F,G)$ is not empty, then there is some $Q\subseteq F\cap G$, hence $F\cap G\neq\varnothing$, and thus $G\subseteq F$ or $F\subsetneq G$. By symmetry, we concentrate on the first case only.

Consider a $Q$ with $\pi(Q)=(F,G)$ for some $G\subseteq F$. Then
\begin{equation}\label{eq:intQg}
  \int_Q g\omega=\int_{Q\cap E_{\mathscr{F}}(F)}g\omega+\sum_{F'\in\textup{ch}_{\mathscr{F}}(F)}\int_{Q\cap F'}g\omega.
\end{equation}
If $Q\cap F'\neq\varnothing$, then either $F'\subsetneq Q$ or $Q\subseteq F'$. But the latter is not possible, since it would imply that $\pi_{\mathscr{F}}(Q)\subseteq F'\subsetneq F$, contradicting $\pi(Q)=(F,G)$. Thus, for the nonzero terms in the last summation in \eqref{eq:intQg}, we must have $F'\subsetneq Q\subseteq G$ for some $G\in\mathscr{F}$ with $G\in\mathscr{G}$. Since $Q\subseteq G\subseteq F$ and $\pi_{\mathscr{F}}Q=F$, also $\pi_{\mathscr{F}}G=F$. Thus, we may actually restrict the summation to
\begin{equation*}
  \textup{ch}_{\mathscr{F}}^*(F) :=\{F'\in\textup{ch}_{\mathscr{F}}(F):\pi_{\mathscr{F}}\pi_{\mathscr{G}}(F')=F\}.
\end{equation*}
So in fact
\begin{equation*}
\begin{split}
  \int_Q g\omega &=\int_{Q\cap E_{\mathscr{F}}(F)}g\omega+\sum_{\substack{F'\in\textup{ch}_{\mathscr{F}}^*(F)\\ F'\subsetneq Q}}\int_{F'}g\omega \\
   &=\int_Q\Big(g 1_{E_{\mathscr{F}}(F)}+\sum_{F'\in\textup{ch}_{\mathscr{F}}^*(F)}\ave{g}_{F'}^{\omega}1_{F'}\Big)\omega
     =:\int_Q g_F\omega.
\end{split}
\end{equation*}
Thus we find that
\begin{equation*}
\begin{split}
  \sum_{G\subseteq F} &\sum_{\substack{Q:\\ \pi(Q)=(F,G)}}
    \lambda_Q\ave{f}_Q^\sigma\ave{\sigma}_Q\int_Q g\omega 
  \leq 2\ave{f}_F^\sigma\sum_{G\subseteq F}\sum_{\substack{Q:\\ \pi(Q)=(F,G)}}\lambda_Q\ave{\sigma}_Q\int_Q g_F\omega \\
  &\leq 2\ave{f}_F^\sigma\sum_{Q\subseteq F}\lambda_Q\ave{\sigma}_Q\int_Q g_F\omega 
  =2\ave{f}_F^\sigma\pair{T_F(\sigma)}{g_F\omega}  \\
  &\leq 2\ave{f}_F^\sigma\Norm{T_F(\sigma)}{L^q(\omega)}\Norm{g_F}{L^{q'}(\omega)} 
  \leq 2\ave{f}_F^\sigma\mathfrak{T}\sigma(F)^{1/p}\Norm{g_F}{L^{q'}(\omega)},
\end{split}
\end{equation*}
and hence
\begin{equation}\label{eq:GsubFest}
\begin{split}
  \sum_F &\sum_{G\subseteq F}\sum_{\substack{Q:\\ \pi(Q)=(F,G)}}\lambda_Q\ave{f}_Q^\sigma\ave{\sigma}_Q\int_Q g\omega 
  \leq 2\mathfrak{T}\sum_F\ave{f}_F^\sigma\sigma(F)^{1/p}\Norm{g_F}{L^{q'}(\omega)} \\
  &\leq 2\mathfrak{T}\Big(\sum_F(\ave{f}_F^\sigma)^p\sigma(F)\Big)^{1/p}\Big(\sum_F\Norm{g_F}{L^{q'}(\omega)}^{p'}\Big)^{1/p'}.
\end{split}
\end{equation}
For the first factor, using $\sigma(F)\leq 2\sigma(E_{\mathscr{F}}(F))$ and $\ave{f}_F^\sigma\leq\inf_F M_\sigma f$, and the disjointness of the $E_{\mathscr{F}}(F)$,
we see that
\begin{equation}\label{eq:fEst}
  \Big(\sum_F(\ave{f}_F^\sigma)^p\sigma(F)\Big)^{1/p}
  \leq 2\Big(\int (M_{\sigma}f)^p\sigma\Big)^{1/p}
    \leq 2 p'\Norm{f}{L^p(\sigma)}.
\end{equation}

Using $\Norm{\ }{\ell^{p'}}\leq\Norm{\ }{\ell^{q'}}$ for $q\geq p$, it only remains to estimate
\begin{equation}\label{eq:gEstStart}
  \Big(\sum_F\Norm{g_F}{L^{q'}(\omega)}^{q'}\Big)^{1/q'}
  =\Big(\sum_F\Norm{g1_{E_{\mathscr{F}}(F)}}{L^{q'}(\omega)}^{q'}+\sum_F\sum_{F'\in\textup{ch}_{\mathscr{F}}^*(F)}
      (\ave{g}_{F'}^\omega)^{q'}\omega(F')\Big)^{1/q'}.
\end{equation}
By the pairwise disjointness of the test $E_{\mathscr{F}}(F)$, it is immediate that
\begin{equation*}
  \sum_F\Norm{g1_{E_{\mathscr{F}}(F)}}{L^{q'}(\omega)}^{q'}
  \leq\Norm{g}{L^{q'}(\omega)}^{q'}.
\end{equation*}
For the remaining double sum, we use the definition of $\textup{ch}_{\mathscr{F}}^*(F)$ to reorganize:
\begin{equation*}
\begin{split}
  \sum_F &\sum_{F'\in\textup{ch}_{\mathscr{F}}^*(F)}(\ave{g}_{F'}^\omega)^{q'}\omega(F')
  =\sum_F\sum_{\substack{G:\\ \pi_{\mathscr{F}}G=F}}\sum_{\substack{F':\\ \pi(F')=(F,G)}}(\ave{g}_{F'}^\omega)^{q'}\omega(F') \\ 
  &\leq \sum_F\sum_{\substack{G:\\ \pi_{\mathscr{F}}G=F}} (2\ave{g}_G^\omega)^{q'}\sum_{\substack{F':\\ \pi(F')=(F,G)}}\omega(F') 
  \leq  2^{q'}\sum_F\sum_{\substack{G:\\ \pi_{\mathscr{F}}G=F}} (\ave{g}_G^\omega)^{q'}\omega(G) \\  
  &\leq  2^{1+q'}\sum_F\sum_{\substack{G:\\ \pi_{\mathscr{F}}G=F}} (\ave{g}_G^\omega)^{q'}\omega(E_{\mathscr{G}}(G)) 
  =  2^{1+q'}\sum_G (\ave{g}_G^\omega)^{q'}\omega(E_{\mathscr{G}}(G)) \\  
  &\leq 2^{1+q'}\int (M_\omega g)^{q'}\omega
    \leq 2^{1+q'}(q\Norm{g}{L^{q'}(\omega)})^{q'}.
\end{split}
\end{equation*}
Substituting back to \eqref{eq:gEstStart}, we have that
\begin{equation}\label{eq:gEstFinish}
\begin{split}
  \Big(\sum_F\Norm{g_F}{L^{q'}(\omega)}^{q'}\Big)^{1/q'}
  \leq\Norm{g}{L^{q'}(\omega)}+2^{1/q'+1}q\Norm{g}{L^{q'}(\omega)}
  \leq 5q\Norm{g}{L^{q'}(\omega)}.
\end{split}
\end{equation}
The combination of \eqref{eq:GsubFest}, \eqref{eq:fEst} and \eqref{eq:gEstFinish} shows that the first half of $\pair{T(f\sigma)}{g\omega}$, according to the splitting \eqref{eq:TfgSplit} is estimated by
\begin{equation*}
    2\mathfrak{T}\cdot 2p'\Norm{f}{L^p(\sigma)}\cdot 5q\Norm{g}{L^{q'}(\omega)}
      \leq 20\cdot\mathfrak{T}\cdot p' q\cdot \Norm{f}{L^p(\sigma)}\Norm{g}{L^{q'}(\omega)}
\end{equation*}
We conclude by symmetry of the assumptions and the splitting \eqref{eq:TfgSplit}.
\end{proof}


\bibliography{weighted}
\bibliographystyle{plain}

\end{document}